\pdfoutput=1
\UseRawInputEncoding
\documentclass[12pt]{amsart}
\usepackage{oldgerm}
\usepackage{latexsym}
\usepackage{amsmath}
\usepackage{amscd}
\usepackage{amssymb} 
\usepackage{amsmath}
\usepackage{amsthm}
\usepackage{latexsym}
\usepackage[mathscr]{eucal}
\makeatletter
\@namedef{subjclassname@2020}{%
  \textup{2020} Mathematics Subject Classification}
\makeatother
\begin{document}
\title [Fourier coefficients of the Duke-Imamoglu-Ikeda lift]{Estimates for the Fourier coefficients of the Duke-Imamoglu-Ikeda lift
}  
\author{Tamotsu IKEDA}
\address{Graduate school of mathematics, Kyoto University, Kitashirakawa, Kyoto, 606-8502, Japan}
\email{ikeda@math.kyoto-u.ac.jp}
\author{Hidenori KATSURADA}
\address{Department of Mathematics, Hokkaido, University, Kita 10, Nishi 8, Kitaku, Sapporo, Hokkaido, 060-0810, Japan, and 
Muroran Institute of Technology
27-1 Mizumoto, Muroran, 050-8585, Japan}
\email{hidenori@mmm.muroran-it.ac.jp}
\thanks{The research was partially supported by JSPS KAKENHI Grant Number 17H02834, 16H03919, 22K03228 and 21K03152.
}

\date {July 5, 2022}

\maketitle
\newcommand{\alp}{\alpha}
\newcommand{\bet}{\beta}
\newcommand{\gam}{\gamma}
\newcommand{\del}{\delta}
\newcommand{\eps}{\epsilon}
\newcommand{\zet}{\zeta}
\newcommand{\tht}{\theta}
\newcommand{\iot}{\iota}
\newcommand{\kap}{\kappa}
\newcommand{\lam}{\lambda}
\newcommand{\sig}{\sigma}
\newcommand{\ups}{\upsilon}
\newcommand{\ome}{\omega}
\newcommand{\vep}{\varepsilon}
\newcommand{\vth}{\vartheta}
\newcommand{\vpi}{\varpi}
\newcommand{\vrh}{\varrho}
\newcommand{\vsi}{\varsigma}
\newcommand{\vph}{\varphi}
\newcommand{\Gam}{\Gamma}
\newcommand{\Del}{\Delta}
\newcommand{\Tht}{\Theta}
\newcommand{\Lam}{\Lambda}
\newcommand{\Sig}{\Sigma}
\newcommand{\Ups}{\Upsilon}
\newcommand{\Ome}{\Omega}


\newcommand{\frka}{{\mathfrak a}}    \newcommand{\frkA}{{\mathfrak A}}
\newcommand{\frkb}{{\mathfrak b}}    \newcommand{\frkB}{{\mathfrak B}}
\newcommand{\frkc}{{\mathfrak c}}    \newcommand{\frkC}{{\mathfrak C}}
\newcommand{\frkd}{{\mathfrak d}}    \newcommand{\frkD}{{\mathfrak D}}
\newcommand{\frke}{{\mathfrak e}}    \newcommand{\frkE}{{\mathfrak E}}
\newcommand{\frkf}{{\mathfrak f}}    \newcommand{\frkF}{{\mathfrak F}}
\newcommand{\frkg}{{\mathfrak g}}    \newcommand{\frkG}{{\mathfrak G}}
\newcommand{\frkh}{{\mathfrak h}}    \newcommand{\frkH}{{\mathfrak H}}
\newcommand{\frki}{{\mathfrak i}}    \newcommand{\frkI}{{\mathfrak I}}
\newcommand{\frkj}{{\mathfrak j}}    \newcommand{\frkJ}{{\mathfrak J}}
\newcommand{\frkk}{{\mathfrak k}}    \newcommand{\frkK}{{\mathfrak K}}
\newcommand{\frkl}{{\mathfrak l}}    \newcommand{\frkL}{{\mathfrak L}}
\newcommand{\frkm}{{\mathfrak m}}    \newcommand{\frkM}{{\mathfrak M}}
\newcommand{\frkn}{{\mathfrak n}}    \newcommand{\frkN}{{\mathfrak N}}
\newcommand{\frko}{{\mathfrak o}}    \newcommand{\frkO}{{\mathfrak O}}
\newcommand{\frkp}{{\mathfrak p}}    \newcommand{\frkP}{{\mathfrak P}}
\newcommand{\frkq}{{\mathfrak q}}    \newcommand{\frkQ}{{\mathfrak Q}}
\newcommand{\frkr}{{\mathfrak r}}    \newcommand{\frkR}{{\mathfrak R}}
\newcommand{\frks}{{\mathfrak s}}    \newcommand{\frkS}{{\mathfrak S}}
\newcommand{\frkt}{{\mathfrak t}}    \newcommand{\frkT}{{\mathfrak T}}
\newcommand{\frku}{{\mathfrak u}}    \newcommand{\frkU}{{\mathfrak U}}
\newcommand{\frkv}{{\mathfrak v}}    \newcommand{\frkV}{{\mathfrak V}}
\newcommand{\frkw}{{\mathfrak w}}    \newcommand{\frkW}{{\mathfrak W}}
\newcommand{\frkx}{{\mathfrak x}}    \newcommand{\frkX}{{\mathfrak X}}
\newcommand{\frky}{{\mathfrak y}}    \newcommand{\frkY}{{\mathfrak Y}}
\newcommand{\frkz}{{\mathfrak z}}    \newcommand{\frkZ}{{\mathfrak Z}}


\newcommand{\bfa}{{\mathbf a}}    \newcommand{\bfA}{{\mathbf A}}
\newcommand{\bfb}{{\mathbf b}}    \newcommand{\bfB}{{\mathbf B}}
\newcommand{\bfc}{{\mathbf c}}    \newcommand{\bfC}{{\mathbf C}}
\newcommand{\bfd}{{\mathbf d}}    \newcommand{\bfD}{{\mathbf D}}
\newcommand{\bfe}{{\mathbf e}}    \newcommand{\bfE}{{\mathbf E}}
\newcommand{\bff}{{\mathbf f}}    \newcommand{\bfF}{{\mathbf F}}
\newcommand{\bfg}{{\mathbf g}}    \newcommand{\bfG}{{\mathbf G}}
\newcommand{\bfh}{{\mathbf h}}    \newcommand{\bfH}{{\mathbf H}}
\newcommand{\bfi}{{\mathbf i}}    \newcommand{\bfI}{{\mathbf I}}
\newcommand{\bfj}{{\mathbf j}}    \newcommand{\bfJ}{{\mathbf J}}
\newcommand{\bfk}{{\mathbf k}}    \newcommand{\bfK}{{\mathbf K}}
\newcommand{\bfl}{{\mathbf l}}    \newcommand{\bfL}{{\mathbf L}}
\newcommand{\bfm}{{\mathbf m}}    \newcommand{\bfM}{{\mathbf M}}
\newcommand{\bfn}{{\mathbf n}}    \newcommand{\bfN}{{\mathbf N}}
\newcommand{\bfo}{{\mathbf o}}    \newcommand{\bfO}{{\mathbf O}}
\newcommand{\bfp}{{\mathbf p}}    \newcommand{\bfP}{{\mathbf P}}
\newcommand{\bfq}{{\mathbf q}}    \newcommand{\bfQ}{{\mathbf Q}}
\newcommand{\bfr}{{\mathbf r}}    \newcommand{\bfR}{{\mathbf R}}
\newcommand{\bfs}{{\mathbf s}}    \newcommand{\bfS}{{\mathbf S}}
\newcommand{\bft}{{\mathbf t}}    \newcommand{\bfT}{{\mathbf T}}
\newcommand{\bfu}{{\mathbf u}}    \newcommand{\bfU}{{\mathbf U}}
\newcommand{\bfv}{{\mathbf v}}    \newcommand{\bfV}{{\mathbf V}}
\newcommand{\bfw}{{\mathbf w}}    \newcommand{\bfW}{{\mathbf W}}
\newcommand{\bfx}{{\mathbf x}}    \newcommand{\bfX}{{\mathbf X}}
\newcommand{\bfy}{{\mathbf y}}    \newcommand{\bfY}{{\mathbf Y}}
\newcommand{\bfz}{{\mathbf z}}    \newcommand{\bfZ}{{\mathbf Z}}


\newcommand{\cala}{{\mathcal A}}
\newcommand{\calb}{{\mathcal B}}
\newcommand{\calc}{{\mathcal C}}
\newcommand{\cald}{{\mathcal D}}
\newcommand{\cale}{{\mathcal E}}
\newcommand{\calf}{{\mathcal F}}
\newcommand{\calg}{{\mathcal G}}
\newcommand{\calh}{{\mathcal H}}
\newcommand{\cali}{{\mathcal I}}
\newcommand{\calj}{{\mathcal J}}
\newcommand{\calk}{{\mathcal K}}
\newcommand{\call}{{\mathcal L}}
\newcommand{\calm}{{\mathcal M}}
\newcommand{\caln}{{\mathcal N}}
\newcommand{\calo}{{\mathcal O}}
\newcommand{\calp}{{\mathcal P}}
\newcommand{\calq}{{\mathcal Q}}
\newcommand{\calr}{{\mathcal R}}
\newcommand{\cals}{{\mathcal S}}
\newcommand{\calt}{{\mathcal T}}
\newcommand{\calu}{{\mathcal U}}
\newcommand{\calv}{{\mathcal V}}
\newcommand{\calw}{{\mathcal W}}
\newcommand{\calx}{{\mathcal X}}
\newcommand{\caly}{{\mathcal Y}}
\newcommand{\calz}{{\mathcal Z}}


\newcommand{\scra}{{\mathscr A}}
\newcommand{\scrb}{{\mathscr B}}
\newcommand{\scrc}{{\mathscr C}}
\newcommand{\scrd}{{\mathscr D}}
\newcommand{\scre}{{\mathscr E}}
\newcommand{\scrf}{{\mathscr F}}
\newcommand{\scrg}{{\mathscr G}}
\newcommand{\scrh}{{\mathscr H}}
\newcommand{\scri}{{\mathscr I}}
\newcommand{\scrj}{{\mathscr J}}
\newcommand{\scrk}{{\mathscr K}}
\newcommand{\scrl}{{\mathscr L}}
\newcommand{\scrm}{{\mathscr M}}
\newcommand{\scrn}{{\mathscr N}}
\newcommand{\scro}{{\mathscr O}}
\newcommand{\scrp}{{\mathscr P}}
\newcommand{\scrq}{{\mathscr Q}}
\newcommand{\scrr}{{\mathscr R}}
\newcommand{\scrs}{{\mathscr S}}
\newcommand{\scrt}{{\mathscr T}}
\newcommand{\scru}{{\mathscr U}}
\newcommand{\scrv}{{\mathscr V}}
\newcommand{\scrw}{{\mathscr W}}
\newcommand{\scrx}{{\mathscr X}}
\newcommand{\scry}{{\mathscr Y}}
\newcommand{\scrz}{{\mathscr Z}}


\newcommand{\AAA}{{\mathbb A}} 
\newcommand{\BB}{{\mathbb B}}
\newcommand{\CC}{{\mathbb C}}
\newcommand{\DD}{{\mathbb D}}
\newcommand{\EE}{{\mathbb E}}
\newcommand{\FF}{{\mathbb F}}
\newcommand{\GG}{{\mathbb G}}
\newcommand{\HH}{{\mathbb H}}
\newcommand{\II}{{\mathbb I}}
\newcommand{\JJ}{{\mathbb J}}
\newcommand{\KK}{{\mathbb K}}
\newcommand{\LL}{{\mathbb L}}
\newcommand{\MM}{{\mathbb M}}
\newcommand{\NN}{{\mathbb N}}
\newcommand{\OO}{{\mathbb O}}
\newcommand{\PP}{{\mathbb P}}
\newcommand{\QQ}{{\mathbb Q}}
\newcommand{\RR}{{\mathbb R}}
\newcommand{\SSS}{{\mathbb S}} 
\newcommand{\TT}{{\mathbb T}}
\newcommand{\UU}{{\mathbb U}}
\newcommand{\VV}{{\mathbb V}}
\newcommand{\WW}{{\mathbb W}}
\newcommand{\XX}{{\mathbb X}}
\newcommand{\YY}{{\mathbb Y}}
\newcommand{\ZZ}{{\mathbb Z}}


\newcommand{\tta}{\hbox{\tt a}}    \newcommand{\ttA}{\hbox{\tt A}}
\newcommand{\ttb}{\hbox{\tt b}}    \newcommand{\ttB}{\hbox{\tt B}}
\newcommand{\ttc}{\hbox{\tt c}}    \newcommand{\ttC}{\hbox{\tt C}}
\newcommand{\ttd}{\hbox{\tt d}}    \newcommand{\ttD}{\hbox{\tt D}}
\newcommand{\tte}{\hbox{\tt e}}    \newcommand{\ttE}{\hbox{\tt E}}
\newcommand{\ttf}{\hbox{\tt f}}    \newcommand{\ttF}{\hbox{\tt F}}
\newcommand{\ttg}{\hbox{\tt g}}    \newcommand{\ttG}{\hbox{\tt G}}
\newcommand{\tth}{\hbox{\tt h}}    \newcommand{\ttH}{\hbox{\tt H}}
\newcommand{\tti}{\hbox{\tt i}}    \newcommand{\ttI}{\hbox{\tt I}}
\newcommand{\ttj}{\hbox{\tt j}}    \newcommand{\ttJ}{\hbox{\tt J}}
\newcommand{\ttk}{\hbox{\tt k}}    \newcommand{\ttK}{\hbox{\tt K}}
\newcommand{\ttl}{\hbox{\tt l}}    \newcommand{\ttL}{\hbox{\tt L}}
\newcommand{\ttm}{\hbox{\tt m}}    \newcommand{\ttM}{\hbox{\tt M}}
\newcommand{\ttn}{\hbox{\tt n}}    \newcommand{\ttN}{\hbox{\tt N}}
\newcommand{\tto}{\hbox{\tt o}}    \newcommand{\ttO}{\hbox{\tt O}}
\newcommand{\ttp}{\hbox{\tt p}}    \newcommand{\ttP}{\hbox{\tt P}}
\newcommand{\ttq}{\hbox{\tt q}}    \newcommand{\ttQ}{\hbox{\tt Q}}
\newcommand{\ttr}{\hbox{\tt r}}    \newcommand{\ttR}{\hbox{\tt R}}
\newcommand{\tts}{\hbox{\tt s}}    \newcommand{\ttS}{\hbox{\tt S}}
\newcommand{\ttt}{\hbox{\tt t}}    \newcommand{\ttT}{\hbox{\tt T}}
\newcommand{\ttu}{\hbox{\tt u}}    \newcommand{\ttU}{\hbox{\tt U}}
\newcommand{\ttv}{\hbox{\tt v}}    \newcommand{\ttV}{\hbox{\tt V}}
\newcommand{\ttw}{\hbox{\tt w}}    \newcommand{\ttW}{\hbox{\tt W}}
\newcommand{\ttx}{\hbox{\tt x}}    \newcommand{\ttX}{\hbox{\tt X}}
\newcommand{\tty}{\hbox{\tt y}}    \newcommand{\ttY}{\hbox{\tt Y}}
\newcommand{\ttz}{\hbox{\tt z}}    \newcommand{\ttZ}{\hbox{\tt Z}}

\newcommand{\phm}{\phantom}
\newcommand{\ds}{\displaystyle }
\newcommand{\smallstrut}{\vphantom{\vrule height 3pt }}
\def\bdm #1#2#3#4{\left(
\begin{array} {c|c}{\ds{#1}}
 & {\ds{#2}} \\ \hline
{\ds{#3}\vphantom{\ds{#3}^1}} &  {\ds{#4}}
\end{array}
\right)}
\newcommand{\wtd}{\widetilde }
\newcommand{\bsl}{\backslash }
\newcommand{\GL}{{\mathrm{GL}}}
\newcommand{\SL}{{\mathrm{SL}}}
\newcommand{\GSp}{{\mathrm{GSp}}}
\newcommand{\PGSp}{{\mathrm{PGSp}}}
\newcommand{\SP}{{\mathrm{Sp}}}
\newcommand{\SO}{{\mathrm{SO}}}
\newcommand{\SU}{{\mathrm{SU}}}
\newcommand{\Ind}{\mathrm{Ind}}
\newcommand{\Hom}{{\mathrm{Hom}}}
\newcommand{\Ad}{{\mathrm{Ad}}}
\newcommand{\Sym}{{\mathrm{Sym}}}
\newcommand{\Mat}{\mathrm{M}}
\newcommand{\sgn}{\mathrm{sgn}}
\newcommand{\trs}{\,^t\!}
\newcommand{\iu}{\sqrt{-1}}
\newcommand{\oo}{\hbox{\bf 0}}
\newcommand{\ono}{\hbox{\bf 1}}
\newcommand{\smallcirc}{\lower .3em \hbox{\rm\char'27}\!}
\newcommand{\bAf}{\bA_{\hbox{\eightrm f}}}
\newcommand{\thalf}{{\textstyle{\frac12}}}
\newcommand{\shp}{\hbox{\rm\char'43}}
\newcommand{\Gal}{\operatorname{Gal}}

\newcommand{\bdel}{{\boldsymbol{\delta}}}
\newcommand{\bchi}{{\boldsymbol{\chi}}}
\newcommand{\bgam}{{\boldsymbol{\gamma}}}
\newcommand{\bome}{{\boldsymbol{\omega}}}
\newcommand{\bpsi}{{\boldsymbol{\psi}}}
\newcommand{\GK}{\mathrm{GK}}
\newcommand{\EGK}{\mathrm{EGK}}
\newcommand{\ord}{\mathrm{ord}}
\newcommand{\diag}{\mathrm{diag}}
\newcommand{\ua}{{\underline{a}}}
\newcommand{\ZZn}{\ZZ_{\geq 0}^n}

\theoremstyle{plain}
\newtheorem{theorem}{Theorem}[section]
\newtheorem{lemma}[theorem]{Lemma}
\newtheorem{proposition}[theorem]{Proposition}
\newtheorem{corollary}[theorem]{Corollary}
\newtheorem{conjecture}[theorem]{Conjecture}
\newtheorem{definition}[theorem]{Definition}
\newtheorem{remark}[theorem]{{\bf Remark}}
\newcommand{\supp}{\mathrm{supp}}

%

\def\mattwono(#1;#2;#3;#4){\begin{array}{cc}
                               #1  & #2 \\
                               #3  & #4
                                      \end{array}}

\def\mattwo(#1;#2;#3;#4){\left(\begin{matrix}
                               #1 & #2 \\
                               #3  & #4
                                      \end{matrix}\right)}
 \def\smallmattwo(#1;#2;#3;#4){\left(\begin{smallmatrix}
                               #1 & #2 \\
                               #3  & #4
                                      \end{smallmatrix}\right)}                                     
                                      
 \def\matthree(#1;#2;#3;#4;#5;#6;#7;#8;#9){\left(\begin{matrix}
                               #1 & #2  & #3\\
                               #4  & #5 & #6\\
                               #7  & #8 &#9 
                                      \end{matrix}\right)}                                     
                                      
\def\mattwo(#1;#2;#3;#4){\left(\begin{matrix}
                               #1 & #2 \\
                               #3  & #4
                                      \end{matrix}\right)}  

\def\rowthree(#1;#2;#3){\begin{matrix}
                               #1   \\
                               #2  \\
                               #3
                                      \end{matrix}}  
\def\columnthree(#1;#2;#3){\begin{matrix}
                               #1   &   #2  &  #3
                                      \end{matrix}}  
                                      
\def\rowfive(#1;#2;#3;#4;#5){\begin{array}{lllll}
                               #1   \\
                               #2  \\
                               #3 \\
                               #4 \\
                               #5                              
                                      \end{array}} 

\def\columnfive(#1;#2;#3;#4;#5){\begin{array}{lllll}
                               #1   &   #2  &  #3 & #4 & #5
                                \end{array}}

\def\mattwothree(#1;#2;#3;#4;#5;#6){\begin{matrix}
                               #1 & #2  & #3  \\
                               #4 & #5  & #6
                                      \end{matrix}}  
\def\matthreetwo(#1;#2;#3;#4;#5;#6){\begin{array}{lc}
                               #1  & #2  \\
                               #3  & #4 \\
                               #5  & #6
                                      \end{array}}  
\def\columnthree(#1;#2;#3){\begin{matrix}
                               #1 & #2 & #3  
                                  \end{matrix}}  
\def\rowthree(#1;#2;#3){\begin{matrix}
                               #1 \\
                                #2 \\
                                #3  
                                  \end{matrix}}  
\def\smallddots{\mathinner
{\mskip1mu\raise3pt\vbox{\kern7pt\hbox{.}}
\mskip1mu\raise0pt\hbox{.}
\mskip1mu\raise-3pt\hbox{.}\mskip1mu}}

\begin{abstract} Let $k$ and $n$ be positive even integers. 
For a Hecke eigenform $h$ in the Kohnen plus subspace of weight $k-n/2+1/2$ for $\varGamma_0(4)$, let $I_n(h)$ be the Duke-Imamoglu-Ikeda lift of $h$ to the space of cusp forms of weight $k$ for $Sp_n(\ZZ)$. We then give an estimate of the Fourier coefficients of $I_n(h)$. It is better than the usual Hecke bound for the Fourier coefficients of a Siegel cusp form.

\end{abstract}

\section{Introduction}
Let $\varGamma^{(n)}=Sp_n(\ZZ)$ be the Siegel modular group of genus $n$, and $S_k(\varGamma^{(n)})$ the space of cusp forms of weight $k$ for $\varGamma^{(n)}$. Then we have the following Fourier expansion:
\[F(Z)=\sum_{B} c_F(B) \exp(2\pi \sqrt{-1} \mathrm{tr}(BZ)),\]
where $B$ runs over all positive definite half-integral matrices of size $n$. 
It is an interesting problem to estimate $c_F(B)$. By the standard method we obtain the following estimate, called the Hecke bound, of $c_F(B)$ for $B \in \calh_n(\ZZ)_{>0}$:
\[c_F(B) \ll_F (\det (B))^{k/2}.\]
However it is weak in general. In the case $n=1$, we have Deligne's estimate (cf. \cite{De74}), which is best possible, and from now on we consider the case $n \ge 2$. Then there are several improvements of the Hecke bound (cf. \cite{BK93}, \cite{BR88}, \cite{Breulmann96}, \cite{Bringmann}, \cite{F87},\cite{Ki84},\cite{Ko93},\cite{RW89}).  Among others, B\"ocherer and Kohnen \cite{BK93} proved that
\begin{align*}
c_F(B) \ll_{\vep,F} \det (B)^{{k \over 2}-{1 \over 2n} -(1-{1 \over n}) \alpha_n+\vep} \ (\vep>0) \tag{$*$}
\end{align*}
if $k >n+1$. Here
$$\alpha_n=\Big(4(n-1)+4\Big[{n-1 \over 2}\Big]+{2 \over n+2}\Big)^{-1}.$$
In this paper, we improve this bound for the Duke-Imamoglu-Ikeda lift $I_n(h)$ of a cuspidal Hecke eigenform $h$ in $S_{k-n/2+1/2}^+(\varGamma_0(4))$ to $S_k(\varGamma^{(n)})$. (As for a precise definition of the Duke-Imamoglu-Ikeda lift, see Section 3.)
That is, we prove the following estimate:

\bigskip

(Theorem \ref{th.main-result}) We have
\begin{align*}
c_{I_n(h)}(B) \ll_{\vep,I_n(h)} \frkd_B^{-n/4+5/12}(\det (2B))^{(k-1)/2+\vep} \ (\vep>0) \tag{**}
\end{align*}
for any  $B \in \calh_n(\ZZ)_{>0}$,
where $\frkd_B$ is the discriminant of $\QQ(\sqrt{(-1)^{n/2} \det B})/\QQ$.

\bigskip

From the above result, we have
\begin{align*}
c_{I_n(h)}(B)|\ll_{\vep,I_n(h) } (\det (2B))^{(k-1)/2+\vep} \ (\vep>0) \tag{***}
\end{align*}
for any  $B \in \calh_n(\ZZ)_{>0}$.
We note that our estimate is slightly stronger than (*). We  explain how to obtain the estimate (**). By definition, $c_{I_n(h)}(B)$ is expressed in terms of the $|\frkd_B|$-th Fourier coefficient $c_h(|\frkd_B|)$ of $h$, and $\prod_p \widetilde F_p(B,\alpha_p)$, where for a prime number $p$, $\widetilde F_p(B,X)$ is the polynomial in $X$ and $X^{-1}$ defined in \cite{IK22}, and $\alpha_p$ is a certain complex number such that $|\alpha_p|=1$ (cf. Section 3). In view of Corollary \ref{cor.estimate-of-FH}, Theorem \ref{th.H-and-Siegel-series}, we can estimate $\widetilde F_p(B,a)$ for any $a \in \CC$ in purely combinatorial method (cf. Theorem \ref{th.estimate-of-F}), 
and therefore we obtain the following estimate (cf. Theorem \ref{th.refined-estimate} (1)):
 \begin{align*}
&|c_{I_n(h)}(B)| \le |c_h(|\frkd_B|)|\frkf_B^{k-1}
\prod_{i=1}^n \prod_{p | \frkf_B}(1+\frke_{i,B}^{(p)})),
\end{align*}
where $\frkf_B=\sqrt{\det (2B)/|\frkd_B|}$, and $\frke_{i,B}^{(p)}$ is that defined after Remark \ref{rem.unstability-of-Kohnen-plus-space}.
On the other hand, by \cite{CI00}, we obtain a reasonable estimate of $c_h(|\frkd_B|)$. Combining these two estimates, we obtain the estimate (**). We also obtain another estimate for $c_{I_n(h)}(B)$ (cf. Theorem \ref{th.main-result2}). 

It is expected that we can obtain a similar estimate for the Fourier coefficient of the lift constructed in \cite{IY20}.

This paper is organized as follows. In Section 2, we review the Siegel series. In Section 3, we state our main result. In Section 4, we review the Gross-Keating invariant. In Section 5, we give an estimate of $\widetilde F_p(B,\alpha_p)$, and in Section 6, we prove our main result. 

We thank Valentin Blomer for many valuable discussions, by which this paper is motivated. We also thank him for many useful comments, by which our main result has been improved greatly.

{\bf Notation} Let $R$ be a commutative ring. We denote by $R^{\times}$ the group of units in $R$. We denote by $M_{mn}(R)$ the set of $(m,n)$ matrices with entries in $R$, and especially write $M_n(R)=M_{nn}(R)$. 
We often identify an element $a$ of $R$ and the matrix $(a)$ of degree 1 whose component is $a$. If $m$ or $n$ is 0, we understand an element of $M_{mn}(R)$ is the {\it empty matrix} and denote it by $\emptyset$. Let $GL_n(R)$ be the group consisting of all invertible elements of $M_n(R)$, and $\mathrm{Sym}_n(R)$ the set of symmetric matrices of degree $n$ with entries in $R$.  For a  semigroup  $S$ we put $S^{\Box}=\{s^2 \ | \ s \in S \}$.
Let $R$ be an integral domain of characteristic different from $2$, and $K$ its quotient field.  We say that an element $A$ of $\mathrm{Sym}_n(K)$ is non-degenerate if the determinant  $\det A$ of $A$ is non-zero.  For a subset $S$ of $\mathrm{Sym}_n(K)$, we denote by
$S^{{\rm{nd}}}$ the subset of $S$ consisting of non-degenerate matrices. 
We say that
a symmetric matrix $A=(a_{ij})$ of degree $n$ with entries in $K$ is half-integral over $R$ if $a_{ii} \ (i=1,...,n)$ and $2a_{ij} \ (1 \le i \not= j \le n)$ belong to $R$. We denote by $\calh_n(R)$ the set of half-integral matrices of degree $n$ over $R$. 
We note that $\calh_n(R)=\mathrm{Sym}_n(R)$ if $R$ contains the inverse of $2$. 
We denote by $\ZZ_{> 0}$ and $\ZZ_{\ge 0}$ the set of positive integers and the set of non-negative integers, respectively.  
 For an $(m,n)$ matrix $X$ and an $(m,m)$ matrix $A$, we write $A[X] ={}^tXAX$, where $^t X$ denotes the transpose of $X$.
 Let $G$ be a subgroup of $GL_n(K)$. Then we say that two elements $B$ and $B'$ in $\mathrm{Sym}_n(K)$  are $G$-equivalent if there is an element
$g$ of $G$ such that $B'=B[g]$. 
 We denote by $1_m$ the unit matrix of degree $m$ and by $O_{m,n}$ the zero matrix of type $(m,n)$. We sometimes abbreviate $O_{m,n}$ as $O$ if there is no fear of  confusion.
For two square matrices $X$ and $Y$ we write $X \bot Y =\mattwo(X;O;O;Y)$. We often write $x \bot Y$ instead of $(x) \bot Y$ if $(x)$ is  a matrix of degree 1. 
For an $m \times n$ matrix,  $B=(b_{ij})$ and sequences ${\bf i}=(i_1,\ldots,i_r), (j_1,\ldots,j_r)$ of integers such that $1 \le i_1, \ldots, i_r \le m, 1 \le j_1,\ldots,j_r \le n$, we put 
\[B\begin{pmatrix}{\bf i}   \\ {\bf j}\end{pmatrix}=  (b_{i_k,j_l})_{1 \le k,l \le r}.\]

\section{Siegel series}
Let $F$ be a non-archimedean local field of characteristic $0$, and $\frko=\frko_F$ its ring of integers.
The maximal ideal and  the residue field of $\frko$ is denoted by $\frkp$ and $\frkk$, respectively.
We fix a prime element $\vpi$ of $\frko$ once and for all.
The cardinality of $\frkk$ is denoted by $q$.
Let $\mathrm{ord}=\mathrm{ord}_{\frkp}$ denote additive valuation on $F$  normalized so that $\mathrm{ord}(\vpi)=1$. If $a=0$,  We write $\mathrm{ord}(0)=\infty$
and we  make the convention that $\mathrm{ord}(0) > \mathrm{ord}(b)$ for any $b \in F^{\times}$.
We also denote by $|*|_{\frkp}$ denote the valuation on $F$ normalized so that $|\vpi|_{\frkp}=q^{-1}$. 
We put $e_0=\mathrm{ord}_{\frkp}(2)$.
 
For a non-degenerate element $B\in\calh_n(\frko)$, we put $D_B=(-4)^{[n/2]}\det B$.
If $n$ is even, we denote the discriminant ideal of $F(\sqrt{D_B})/F$ by $\frkD_B$.
We also put
\[
\xi_B=
\begin{cases} 
1 & \text{ if $D_B\in F^{\times 2}$,} \\
-1 & \text{ if $F(\sqrt{D_B})/F$ is unramified quadratic,} \\
0 & \text{ if $F(\sqrt{D_B})/F$ is ramified quadratic.} 
\end{cases}
\]
Put 
$$\frke_B=
\begin{cases}
\mathrm{ord}(D_B)-\mathrm{ord}(\frkD_B)   & \text{ if $n$ is even} \\
\mathrm{ord}(D_B)                         & \text{ if $n$ is odd.}
\end{cases}$$

We make the convention that $\xi_B=1, \frke_B=0$  if $B$ is the empty matrix.
Once for all, we fix an additive character $\psi$ of $F$ of order zero, that is, a character such that
$$\frko =\{ a \in F \ | \ \psi(ax)=1 \ \text{ for  any} \ x \in \frko \}.$$  For  a half-integral matrix $B$ of degree $n$ over $\frko$ define the local Siegel series $b_{\frkp}(B,s)$ by 
$$b_{{\frkp}}(B,s)= \sum_{R} \psi(\mathrm{tr}(BR))\mu(R)^{-s},$$
where $R$ runs over a complete set of representatives of $\mathrm{Sym}_n(F)/\mathrm{Sym}_n(\frko)$ and $\mu(R)=[R\frko^n+\frko^n:\frko^n]$. 
The series $b_{\frkp}(B,s)$ converges absolutely if the real part of $s$ is large enough, and it has a meromorphic continuation to the whole $s$-plane. 
Now for a non-degenerate half-integral matrix $B$ of degree $n$ over $\frko $ define a polynomial $\gamma_q(B,X)$ in $X$ by $$\gamma_q(B,X)=
\begin{cases}
(1-X)\prod_{i=1}^{n/2}(1-q^{2i}X^2)(1-q^{n/2}\xi_B X)^{-1} & \text{ if $n$ is  even} \\ 
(1-X)\prod_{i=1}^{(n-1)/2}(1-q^{2i}X^2) & \text{ if $n$ is  odd.} \end{cases}$$ 
Then it is shown by \cite{Sh1} that there exists a polynomial $F_{\frkp}(B,X)$ in $X$ such that 
$$F_{\frkp}(B,q^{-s})={b_{\frkp}(B,s) \over \gamma_q(B,q^{-s})}.$$ 
 
We define a symbol $X^{1/2}$ so that $(X^{1/2})^2=X$.
We define $\widetilde F_{\frkp}(B,X)$ as
$$\widetilde F_{\frkp}(B,X)=X^{-\frke_B/2}F(B,q^{-(n+1)/2}X).$$
We note that $\widetilde F_{\frkp}(B,X) \in \QQ(q^{1/2})[X,X^{-1}]$ if $n$ is even, and
$\widetilde F_{\frkp}(B,X) \in \QQ[X^{1/2},X^{-1/2}]$ if $n$ is odd.
\section{The Duke-Imamoglu-Ikeda lift and main result}
Put $J_n=\begin{pmatrix}O_n&-1_n\\1_n&O_n\end{pmatrix}$.
 Furthermore, put 
$$\varGamma^{(n)}=Sp_n({\ZZ})=\{M \in GL_{2n}({\ZZ})   \ | \  J_n[M]=J_n \}.
$$
Let ${\Bbb H}_n$ be Siegel's
upper half-space of degree $n$. We define $j(\gamma,Z)=\det (CZ+D)$ for $\gamma = \begin{pmatrix} A & B \\ C & D \end{pmatrix}$ and $Z \in {\Bbb H}_n$. We note that $\varGamma^{(1)}=SL_2(\ZZ)$. Let $l$ be an integer or half-integer. For a congruence subgroup $\varGamma$ of $\varGamma^{(n)}$, we denote by $M_{l}(\varGamma)$ the space of Siegel modular forms of weight $l$ with respect to $\varGamma$,  and by $S_{l}(\varGamma)$ its subspace consisting of cusp forms.
Let $T$ be an element of ${\calh_n}(\ZZ)_{>0}$ with $n$ even. Let $\frkd_T$ be the discriminant of $\QQ(\sqrt{(-1)^{n/2} \det (T)})/\QQ$. Then  
we have $(-1)^{n/2} \det (2T)/\frkd_T=\frkf_T^2$ with $\frkf_T \in \ZZ_{>0}$.
Now let $k$ be a positive even integer, and $\varGamma_0(4)=\Bigl\{\bigl( \begin{smallmatrix} a & b \\ c & d \end{smallmatrix} \bigr) \in SL_2(\ZZ) \ | \ c \equiv 0 \text{ mod} 4. \bigr\}$. Let 
 $$h(z)=\sum_{m \in {\ZZ}_{>0} \atop (-1)^{n/2}m \equiv 0, 1 \ {\rm mod} \ 4 }c_h(m){\bf e}(mz)$$
  be a Hecke eigenform in the Kohnen plus space $S_{k-n/2+1/2}^+(\varGamma_0(4))$ and $f(z)=\sum_{m=1}^{\infty}c_f(m){\bf e}(mz)$ be 
 the primitive form in $S_{2k-n}(SL_2(\Bbb Z))$ corresponding to $h$ under the Shimura correspondence (cf. Kohnen \cite{Ko}).
We define a Fourier series $I_n(h)(Z)$ in $Z \in {\Bbb H}_n$ by
$$I_n(h)(Z)= \sum_{T \in {\calh_n}(\ZZ)_{> 0}} c_{I_n(h)}(T){\bf e}({\rm tr}(TZ))$$
where $c_{I_n(h)}(T)=c_h(|{\textfrak d}_T|) {\textfrak f}_T^{k-(n+1)/2} \prod_p\widetilde F_p(T,\alpha_f(p)).
$
Then the first named author \cite{I01} showed that $I_n(h)(Z)$ is a Hecke eigenform in $S_k(\varGamma^{(n)})$ whose
standard $L$-function coincides with $\zeta(s)\prod_{i=1}^n L(s+k-i,f)$ (see also \cite{IY20}).
We call $I_n(h)$ the Duke-Imamoglu-Ikeda lift (D-I-I lift for short) of $h$.
\begin{theorem} \label{th.main-result}
Let the notation be as above.  Then we have 
\begin{align*}
c_{I_n(h)}(B) \ll_{\vep,I_n(h)} \frkd_B^{-n/4+5/12}(\det (2B))^{(k-1)/2+\vep} \ (\vep >0)
\end{align*}
for any $B \in \calh_n(\ZZ)_{>0}$.
In particular, we have 
\begin{align*}
c_{I_n(h)}(B) \ll_{\vep,I_n(h)} (\det (2B))^{(k-1)/2+\vep} \ (\vep >0)
\end{align*}
for any $B \in \calh_n(\ZZ)_{>0}$.
\end{theorem}
We give another estimate of $c_{I_n(h)}(B)$ in terms of minors of $B$. 
We  denote by  $\cali_r$ the set of sequences $(i_1,\ldots,i_r)$ of integers such that $1 \le i_1<\cdots<i_r \le n$.
Let $R$ be an integral domain of characteristic different from $2$. 
 For an element $B =(b_{ij}) \in \calh_n(R)$, and ${\bf i}=(i_1,\ldots,i_r), {\bf j}=(j_1,\ldots,j_r)  \in \cali_r$  we define  $b_{{\bf i},{\bf j}}^{(r)}=b_{{\bf i},{\bf j}}^{(r)}(B)$ as
\[ b_{{\bf i},{\bf j}}^{(r)}=2^{2[r/2]+1 -\delta_{{\bf i},{\bf j}}}\det B\begin{pmatrix} {\bf i} \\ {\bf j} \end{pmatrix},\]
where
\[\delta_{{\bf i},{\bf j}}=\begin{cases} 1 & \text{ if } {\bf i}={\bf j} \\
0 & \text{ otherwise}.\end{cases}.\]
For $B \in \calh_n(\ZZ)_{>0}$ put \[G_r(B)=\mathrm{GCD}_{({\bf i},{\bf j}) \in \cali_r \times \cali_r} b_{{\bf i},{\bf j}}^{(r)}.\] 
\begin{theorem} \label{th.main-result2}
Let the notation be as above.  Then we have 
\begin{align*}
&c_{I_n(h)}(B) \\
&\ll_{\vep,I_n(h)} \frkd_B^{1/6}(\det (2B))^{k/2-(n+1)/4+\vep}\prod_{i=1}^{n-1}G_i(B)^{1/2} \quad  (\vep >0)
\end{align*}
for any $B \in \calh_n(\ZZ)_{>0}$.
In particular, we have 
\begin{align*}
&c_{I_n(h)}(B) \\
&\ll_{\vep,I_n(h)} (\det (2B))^{k/2-n/4-1/12+\vep} \prod_{i=1}^{n-1}G_i(B)^{1/2}\quad  (\vep >0)
\end{align*}
for any $B \in \calh_n(\ZZ)_{>0}$.
\end{theorem}

\section{The Gross-Keating invariant}
\label{sec:1}

We first recall the definition of the Gross-Keating invariant of a quadratic form over $\frko$ following \cite{IK18}. 

For two matrices $B, B'\in\calh_n(\frko)$, we sometimes write $B\sim B'$ if $B$ and $B'$ are $GL_n(\frko)$-equivalent. 
The $GL_n(\frko)$-equivalence class of $B$ is denoted by $\{B\}$.
Let $B=(b_{ij}) \in \calh_n(\frko)^{\rm nd}$. 
Let $S(B)$ be the set of all non-decreasing sequences $(a_1, \ldots, a_n)\in\ZZn$ such that
\begin{align*}
\ord(b_i)&\geq a_i, \\
\ord(2 b_{ij})&\geq (a_i+a_j)/2\qquad (1\leq i,j\leq n).
\end{align*}
Set
\[
S(\{B\})=\bigcup_{B'\in\{B\}} S(B')=\bigcup_{U\in\GL_n(\frko)} S(B[U]).
\]
The Gross-Keating invariant (or the GK-invariant for short) $\ua=(a_1, a_2, \ldots, a_n)$ of $B$ is the greatest element of $S(\{B\})$ with respect to the lexicographic order $\succ$ on $\ZZn$.
Here, the lexicographic order $\succ$ is, as usual, defined as follows.
For $(y_1, y_2, \ldots, y_n),  (z_1, z_2, \ldots, z_n)\in \ZZ_{\geq 0}^n$, let $j$ be the largest integer such that $y_i=z_i$ for $i<j$.
Then $(y_1, y_2, \ldots, y_n)\succ  (z_1, z_2, \ldots, z_n)$ if $y_j>z_j$.
The Gross-Keating invariant  is denoted by  $\GK(B)$.
A sequence of length $0$ is denoted by $\emptyset$.
When $B$ is a matrix of degree $0$, we understand $\GK(B)=\emptyset$.
By definition, the Gross-Keating invariant $\GK(B)$ is determined only by the $GL_n(\frko)$-equivalence class of $B$.
We say that $B\in\calh_n(\frko)$ is an optimal form if $\GK(B)\in S(B)$.
Let $B \in \calh_n(\frko)$. Then $B$ is $GL_n(\frko)$-equivalent to an optimal form $B'$.
Then we say that $B$ has an optimal decomposition $B'$. 
We say that $B \in \calh_n(\frko)$ is a diagonal Jordan form if $B$ is expressed as 
$$B=\vpi^{a_1} u_1 \bot \cdots \bot \vpi^{a_n}u_n$$ 
with $a_1 \le \cdots \le a_n$ and $u_1,\cdots,u_n \in \frko^{\times}$.
 Then, in the non-dyadic case,  the diagonal Jordan form $B$ above  is optimal, and
$\GK(B)=(a_1,\ldots,a_n)$.
Therefore, the diagonal Jordan decomposition is an optimal decomposition. However, in the dyadic case, not all half-integral symmetric matrices have a diagonal Jordan decomposition, and  the Jordan decomposition is not necessarily an optimal decomposition.
Let $B \in \calh_n(\frko)^{\rm nd}$, and let $\GK(B)=(a_1,\ldots,a_n)$.
For $1 \le i \le n$ put
\[\frke_i =\frke_i(B)=\begin{cases}
a_1+\cdots+a_i & \text{ if } i \text{ is  odd} \\
2[(a_1+\cdots+a_i)/2] & \text{ if } i \text{ is even}
\end{cases}\]
The following result is due to \cite[Theorem 0.1]{IK18}, and plays an important role in proving our main result:
\begin{theorem} \label{th.GK-invariant}
Let $B \in \calh_n(\frko)^{\mathrm{nd}}$. Then
\begin{align*}
\frke_n(B)=\frke_B .
\end{align*}
\end{theorem}
For our later purpose, we give an estimate of $\frke_i$. 
We put
\[{\bf f}_r(B)=\min_{X \in M_{nr}(\frko) \atop \det B[X] \not=0} 
\frke_{B[X]},\]
and
\[{\bf d}_r(B)=\min_{X \in M_{nr}(\frko)} \ord_{\frkp}(\det B[X]).\]
Clearly we have ${\bf f}_r(B) \le {\bf d}_r(B)$.
\begin{theorem}\label{th.estimate-of-GK}
Let $B \in \calh_n(\frko)^{\rm nd}$, and $r \le n-1$ be a positive integer.
Then,
\[\frke_r \le {\bf f}_r(B),\]
and in particular 
\[\frke_r \le {\bf d}_r(B).\]
\end{theorem}
\begin{proof}
The assertion follows from \cite[Lemma 3.8]{IK18}.
\end{proof}
We express ${\bf d}_r(B)$ in terms of minors of $B$.
For a sequence ${\bf i}=(i_1,\ldots,i_r)$ of integers let $\mathrm{supp} ({\bf i})$ denote the set $\{i_1,\ldots,i_r\}$.
For $B \in \calh_n(\frko)$ 
\[{\bf g}_r(B)=\min_{({\bf i},{\bf j}) \in \cali_r \times \cali_r} \ord_{\frkp}(b_{{\bf i},{\bf j}}^{(r)}).\]
For $X=(x_{ij}) \in M_{nr}(R)$, and ${\bf i}=(i_1,\ldots,i_r) \in \cali_r$ let $X({\bf i})=\det X\begin{pmatrix} 1,\ldots, r \\ i_1,\ldots,i_r \end{pmatrix}$.
\begin{lemma} \label{lem.det-of-H[X]}
Let $B \in \calh_n(R)$. Then, for any $X,Y \in M_{nr}(R)$, we have
\begin{align*}
2^{2[r/2]+1-\delta_{X,Y}}\det ({}^tXBY)=\sum_{({\bf i},{\bf j}) \in \cali_r \times \cali_r}  
2^{\delta_{{\bf i},{\bf j}}-\delta_{X,Y}}b_{{\bf i},{\bf j}}^{(r)}
X({\bf i})Y({\bf j}),
\end{align*}
where $\delta_{X,Y}=\begin{cases} 1 & \text{ if } X=Y \\ 0 & \text{otherwise}\\\end{cases}$.
In particular, we have
\begin{align*}
&2^{2[r/2]}\det B[X] =\sum_{({\bf i},{\bf j}) \in \cali_r \times \cali_r \atop
{\bf j}  \succ  {\bf i}} b_{{\bf i},{\bf j}}^{(r)}
X({\bf i})X({\bf j}) 
\end{align*}
\end{lemma}
\begin{proof}
The first assertion can be proved using \cite[II, Theorem 9 (ii)]{Satake75} twice, and the second assertion can be proved remarking that $ \det B\begin{pmatrix} {\bf j} \\ {\bf i} \end{pmatrix} =\det B\begin{pmatrix} {\bf i} \\ {\bf j} \end{pmatrix}$.
\end{proof}
Let $B$ and $B'$ be elements of $\calh_n(\frko)^{\rm nd}$ and suppose that  $B'$ is $GL_n(\frko)$-equivalent to $B$. Then clearly we have ${\bf d}_r(B')={\bf d}_r(B)$. Moreover, by the above lemma, we have ${\bf g}_r(B)={\bf g}_r(B')$.
\begin{theorem}\label{th.estimate-of-GK2}
Let $B \in \calh_n(\frko)^{\rm nd}$. Then any positive integer $r \le n$, we have
\[{\bf d}_r(B)={\bf g}_r(B).\]
\end{theorem}
\begin{proof}
The assertion clearly holds if $r=n$, and therefore we assume $r \le n-1$. By Lemma \ref{lem.det-of-H[X]}, we have ${\bf g}_r(B) \le {\bf d}_r(B)$. 
For each ${\bf i}=(i_1,\ldots,i_r), {\bf j}=(j_1,\ldots,j_r) \in \cali_r$ put
$s({\bf i},{\bf j})=\#(\mathrm{supp}({\bf i}) \cup \mathrm{supp}({\bf j}))$,
and 
\[s(B)=\min_{({\bf i},{\bf j}) \in \cali_r \times \cali_r \atop 
\ord_\frkp(b_{{\bf i},{\bf j}}^{(r)})={\bf g}_r(B)} s({\bf i},{\bf j}).\]
Then we have $r \le s(B) \le \min(n,2r)$.

First let $F$ be a non-dyadic field. In view of the remark before this theorem, we may assume that $B$ is a diagonal matrix. Let $({\bf i},{\bf j}) \in \cali_r \times \cali_r$ such that
$\ord_\frkp(b_{{\bf i},{\bf j}}^{(r)})={\bf g}_r(B)$. Clearly we have $s(B)=r$, and hence ${\bf i}={\bf j}$.
Permutating the rows and columns of $B$ appropriately, we may assume ${\bf i}={\bf j}=(1,\ldots,r)$. Then by  Lemma \ref{lem.det-of-H[X]}, we have
\[B \Big[\begin{pmatrix} 1_r \\ O_{n-r,r} \end{pmatrix}\Big]=b_{{\bf i},{\bf i}}^{(r)}.\]
Hence we have ${\bf d}_r(B) \le {\bf g}_r(B)$. This proves the assertion. 

Next let $F$ be a dyadic field, and put $e_0=\ord_\frkp(2)$ as in Section 2. Then, by \cite[Section 2]{IK18} and \cite[Section 93]{Omeara73}, we may assume that
\begin{align*} B=\vpi^{a_1} K_1 \bot \vpi^{a_{n_1}}K_{n_1} \bot \vpi^{a_{n_1}+1} u_{n_1+1} \bot \vpi^{a_{n_2}} u_{n_2}, \tag{$\bullet$}
\end{align*}
where $a_i \in \ZZ_{\ge 0} \ (i=1,\ldots,n_2), u_i \in \frko^\times \ (i=n_1+1,\ldots,n_2)$ and 
\[K_i=\begin{pmatrix} \alpha_i & \vpi^{f_i}/2 \\\vpi^{f_i}/2 & \beta_i \end{pmatrix}\]
with $\alpha_i  \in \frko^\times, \beta_i \in \frko, 0 \le f_i \le e_0-1$ ($i=1,\ldots,n_1)$.
We claim that $s(B) \le r+1$. 
Suppose that $s(B) \ge r+2$. 
Then, clearly we have $n_1 \ge 2$.
Let $({\bf i},{\bf j}) \in \cali_r \times \cali_r$ such that $s({\bf i},{\bf j})=s(B)$ and 
$\ord_\frkp(b_{{\bf i},{\bf j}}^{(r)})={\bf g}_r(B)$. 
Let $i_k$ be the least integer such that $i_k \in \supp ({\bf i}) \setminus (\supp ({\bf i}) \cap \supp({\bf j}))$. Then, we have $i_k \le 2n_1$. By ($\bullet$),
we have $j_k=i_k+1$ or $j_k=i_k-1$. Without loss of generality, we may assume $j_k=i_k+1$.
Then, $i_k=2i-1$ and $j_k=2i$ with some $1 \le i \le n_1$. Let $i_{l}$ be the least integer such that $i_{l} >i_k$ and $i_{l} \in \supp ({\bf i}) \setminus (\supp ({\bf i}) \cap \supp({\bf j}))$. Again by ($\bullet$), we have
$(i_{l},j_{l})=(2j-1,2j)$ or $(i_{l},j_{l})=(2j,2j-1)$ with some $i<j \le n_1$.
In the former case,
\begin{align*}
\det B \begin{pmatrix} {\bf i} \\ {\bf j} \end{pmatrix}=
&\det B\begin{pmatrix} 2i-1, 2j-1 \\ 2i,2j\end{pmatrix} \det B\begin{pmatrix} {\bf i}'' \\ {\bf j}'' \end{pmatrix} \\
&=4^{-1}\vpi^{a_i+a_j +f_i +f_j}\det B\begin{pmatrix} {\bf i}'' \\ {\bf j}'' \end{pmatrix}
\end{align*}
where $({\bf i}'',{\bf j}'') \in \cali_{r-2} \times \cali_{r-2}$ such that
$\mathrm{supp}({\bf i}'')=\mathrm{supp}({\bf i}) \setminus \{2i-1,2j-1\}$ and
$\mathrm{supp}({\bf j}'')=\mathrm{supp}({\bf j}) \setminus \{2i,2j\}$. Without loss of generality, we may assume $a_i+f_i \le a_j+f_j$. Let $({\bf i}',{\bf j}')$ be an element of $\cali_r \times \cali_r$ such that 
$\mathrm{supp} ({\bf i}')=\mathrm{supp} ({\bf i}'') \cup \{2i,2i-1 \}$.
and 
$\mathrm{supp} ({\bf j}')=\mathrm{supp} ({\bf i}'') \cup \{2i,2i-1 \}$.
Then, $s({\bf i}',{\bf j}')=s({\bf i},{\bf j})-2$ and
\begin{align*}
\det B \begin{pmatrix} {\bf i}' \\ {\bf j}' \end{pmatrix}=
\det B\begin{pmatrix} 2i-1, 2i \\ 2i-1,2i\end{pmatrix} \det B\begin{pmatrix} {\bf i}'' \\ {\bf j}'' \end{pmatrix} 
=\det (\vpi^{a_i} K_i) \det B\begin{pmatrix} {\bf i}'' \\ {\bf j}'' \end{pmatrix} ,
\end{align*}
and hence
\[\ord_\frkp(b_{{\bf i}',{\bf j}'}^{(r)}) \le \ord_\frkp( b_{{\bf i},{\bf j}}^{(r)}).\]
In the latter case, we also obtain a similar inequality.
This is a contradiction, and we prove the claim.
Suppose that $s(B)=r$. Then, in the same way as in the non-dyadic case, we prove
${\bf d}_r(B) \le {\bf g}_r(B)$. 
 Next suppose that
$s(B)=r+1$. Then, we may assume
${\bf i}=(1,\ldots,r)$ and ${\bf j}=(1,\ldots,r-1,r+1)$.
 If $\ord_{\frkp}(b_{{\bf j},{\bf j}}^{(r)})={\bf g}_r(B)$, then the assertion can be proved in the same manner as above, and we may assume that $\ord_{\frkp}(b_{{\bf j},{\bf j}}^{(r)}) >{\bf g}_r(B)$. 
Put $X=\begin{pmatrix} 1_{r-1} & 0 \\ 0 &1 \\ 0 & 1 \\ O_{n-r-1,r-1} & 0
\end{pmatrix}$. 
Then, again by Lemma \ref{lem.det-of-H[X]}, we have
\[2^{2[r/2]}\det B[X]=b_{{\bf i},{\bf i}}^{(r)} +b_{{\bf i},{\bf j}}^{(r)} +b_{{\bf j},{\bf j}}^{(r)},\]
and hence
\[\ord_\frkp(2^{2[r/2]}\det B[X])=\ord_\frkp(b_{{\bf i},{\bf j}}^{(r)}).\]
Hence we have ${\bf d}_r(B) \le {\bf g}_r(B)$ also in this case. 

This completes the assertion.

\end{proof}
\begin{remark}
For $B=(b_{ij}) \in \calh_n(\frko)^{\rm nd}$, we have
\[\frke_1=\min_{1 \le i,j \le n} \ord_\frkp(b_{i,j}^{(1)}).\]
This has been proved in the case $F=\QQ_p$ (cf. \cite{Y04}), and
can be proved in the same manner in the general case.

\end{remark}

\section{Estimate of $\widetilde F_\frkp(B,\alpha)$.}
In this section we estimate $\widetilde F_\frkp(B,X)$ for $B \in \calh_n(\frko)^\mathrm{nd}$ with $n$ even  and $\alpha \in \CC^\times$. This is one of key ingredients in the proof of our main result. 
We recall the definition of a naive $\EGK$ datum (cf. \cite{IK18}).  Let ${\mathcal Z}_3=\{0,1,-1 \}$. 
\begin{definition} \label{def.NEGK}
An element $(a_1,\ldots,a_n;\vep_1,\ldots,\vep_n)$
of $\ZZ_{\ge 0}^n \times {\mathcal Z}_3^n$ is said to be a naive  $\EGK$ datum of length $n$  if the following conditions hold:
\begin{itemize}
\item [(N1)] $a_1 \le \cdots \le a_n$.
\item [(N2)] Assume that $i$ is even. Then $\vep_i \not=0$ if and only if $a_1+\cdots+a_i$ is even.
\item [(N3)] Assume that $i$ is odd. Then $\vep_i \not=0$. 
\item[(N4)]  $\vep_1=1$.
\item[(N5)]  Let $i  \ge 3$ be an odd integer and assume that $a_1+\cdots + a_{i-1}$ is even.  Then $\vep_i=\vep_{i-1}^{a_i+a_{i-1}}\vep_{i-2}$.
\end{itemize}
We denote by  $\mathcal{NEGK}_n$ the set of all naive $\EGK$ data of length $n$. 
\end{definition}
\begin{definition} \label{def.mono-ass-NEGK} 
For integers $e,\widetilde e$,  a real number $\xi$,  define  rational functions $C(e,\widetilde e,\xi;Y,X)$ 
and $D(e,\widetilde e,\xi;Y,X)$ in $Y^{1/2}$ and $X^{1/2}$ by 
\[C(e,\widetilde e,\xi;Y,X)={Y^{\widetilde e/2}X^{-(e- \widetilde e)/2-1}(1-\xi Y^{-1} X)  \over X^{-1}-X} \]
and
\[D(e,\widetilde e,\xi;Y,X)= {Y^{\widetilde e/2}X^{-(e-\widetilde e)/2}   \over 1- \xi X} .\]
\end{definition}
For a positive integer $i$  put 
$$C_i(e,\widetilde e,\xi;Y,X)= \begin{cases}
C(e,\widetilde e,\xi;Y,X) &  \text { if  $i$ is even } \\
D(e,\widetilde e,\xi;Y,X) & \text{ if $i$ is odd.}
\end{cases}.$$
\begin{definition} \label{def.integer-ass-sequence}
For a sequence $\underline a=(a_1,\ldots,a_n)$ of integers and an integer $1 \le i \le n$, we define $\frke_i=\frke_i(\underline a)$ as
$$\frke_i=
\begin{cases} a_1+\cdots +a_i  & \text{ if  $i$ is odd} \\
2[(a_1+\cdots+a_i)/2] & \text{ if $i$ is even.}
\end{cases}$$
 We also put $\frke_0=0$.
\end{definition}
 For a naive $\EGK$ datum $(a_1,\ldots,a_n;\vep_1,\ldots,\vep_n)$ and an integer $1 \le m\le n$, put
$H_m=(a_1,\ldots,a_m;\vep_1,\ldots,\vep_m)$. Then $H_m$ is also a naive $\EGK$ datum of length $m$.
 \begin{definition} \label{def.pol-ass-NEGK}
For a naive $\EGK$ datum $H=(a_1,\ldots,a_n;\vep_1,\ldots,\vep_n)$ we define a rational function $\calf(H;Y,X)$ in $X^{1/2}$ and $Y^{1/2}$ as follows:
First we define
$$\calf(H;Y,X)=X^{-a_1/2}+X^{-a_1/2+1}+\cdots+X^{a_1/2-1}+X^{a_1/2}$$
if $n=1$. Let  $n>1$. Then $H'= (a_1,\ldots,a_{n-1};\vep_1,\ldots,\vep_{n-1})$ is a naive $\EGK$ datum of length $n-1$. 
Assume that $\calf(H';Y,X)$ is defined for $H'$. Then, we define $\calf(H;Y,X)$ as
\begin{align*}
&\calf(H;Y,X)=C_n(\frke_n,\frke_{n-1},\xi;Y,X)\calf(H';Y,YX)\\
&+\zeta C_n(\frke_n,\frke_{n-1},\xi;Y,X^{-1})\calf(H';Y,YX^{-1}),
\end{align*}
where $\xi=\vep_n$ or $\vep_{n-1}$ according as $n$ is even or odd, and $\zeta=1$ or $\vep_n$ according as $n$ is even or odd.
\end{definition}
The following result is due to  \cite[Proposition 4.1]{IK22}.
\begin{proposition}
\label{prop.fc}
Let $H$ be a naive $\EGK$ datum of length $n$. Then we have
\begin{align*}
\calf(H;Y,X^{-1})=\zeta \calf(H;Y,X),
\end{align*}
where $\zeta=\vep_n$ or $1$ according as $n$ is odd or even.
\end{proposition}
For a naive $\EGK$ datum $H$, let  $\calg(H;Y,X)=X^{\frke_n/2}\calf(H;Y,X)$. 
It follows from the proof of \cite[Proposition 4.2]{IK22}, 
$\calg(H;Y,X)$ is a polynomial in $X$ of degree $\frke_{n}$ with coefficients in  $ \QQ[Y,Y^{-1}]$, write
\begin{align*}
\calg(H;Y,X)=\sum_{i=0}^{\frke_n} a_i(H,Y)X^i.
\end{align*}
We give an induction formula for $\calg(H;Y,X)$. 

\begin{theorem}\label{th.induction-formula-for-G} 
Let $H=(a_1,\ldots,a_n;\vep_1,\ldots,\vep_n)$ be a naive $\EGK$ datum of length $n$, and  put $a_i(Y)=a_i(H,Y)$ and $b_i(Y)=a_i(H_{n-1};Y)$.
\begin{itemize}
\item[(1)] Let $n$ be an even integer such that $n \ge 2$, 
Then, for any $l=0,\ldots,\frke_n$, $a_i(Y)$ is expressed as 
\begin{align*}
a_i(Y)=&\sum_{\max((l-\frke_{n-1})/2,0) \le j \le l/2} b_{l-2j}(Y)Y^{l-2j}\\
&-\vep_n\sum_{\max((l-1-\frke_{n-1})/2,0) \le j \le (l-1)/2} b_{l-1-2j}(Y)Y^{l-2-2j}\\
&-\sum_{0 \le j \le (\frke_{n-1}+l-\frke_n-2)/2} b_{\frke_n-l-2j+2}(Y)Y^{\frke_n-l+2j+2}\\
&+\vep_n \sum_{0 \le j \le (\frke_{n-1}+l-\frke_n-1)/2} b_{\frke_n-l-2j+1}(Y)Y^{\frke_n-l+2j}.
\end{align*}
\item[(2)]  Let $n$ be odd such that $n \ge 3$. 
\begin{itemize}
\item[(2.1)] Assume that $\vep_{n-1} \not=0$. Then, for any $l=0,\ldots,\frke_n$, $a_i(Y)$ is expressed as 
\begin{align*}
a_i(Y)& =\sum_{\max(l-\frke_{n-1},0) \le j \le l } b_{l-j}(Y)Y^{l-j}\\
&-\vep_n \sum_{0 \le j \le \frke_{n-1}+l-\frke_n-1} b_{\frke_n-l-2j+2}(Y)Y^{\frke_n-l+j+1}\vep_{n-1}^j.
\end{align*}
\item[(2.2)] Assume that  $\vep_{n-1} =0$. 
Then, for any $l=0,\ldots,\frke_n$, $a_i(Y)$ is expressed as 
\[a_i(Y)=b_l Y^l  +\vep_n b_{\frke_n-l}Y^{\frke_n-l}.\]
\end{itemize}
\end{itemize}
Throughout (1),(2),(3), we make the convention that the sum $\sum_{0 \le j \le a} (*) $ is zero if $a<0$.
We also understand $b_j=0$ if $j<0$ or $j >\frke_{n-1}$.

\end{theorem}
\begin{proof} 
Let $n$ be even. Then 
\begin{align*}
&\calg(H;Y,X) \tag{A}\\
&={(1-Y^{-1}\vep_n X)\calg(H_{n-1}:Y,YX) \over 1-X^2} \\
&-{X^{e_n+2}
(1-Y^{-1}\vep_n X^{-1})\calg(H_{n-1};Y,YX^{-1}) \over 1-X^2}
\end{align*}
Let $\calh(Y,X)$ denote the right-hand side of (A). Then, as a formal power series in $X$, $\calh(Y,X)$ can be written as
\begin{align*}
&\calh(Y,X)\\
&=(1-Y^{-1}\vep_n X)\calg(H_{n-1}:Y,YX) \sum_{j=0}^{\infty} X^{2j}\\
&-X^{e_n+2}
(1-Y^{-1}\vep_n X^{-1})\calg(H_{n-1};Y,YX^{-1})\sum_{j=0}^{\infty} X^{2j}\\
&=(1-Y^{-1}\vep_nX)\sum_{i=0}^{\frke_{n-1}} b_i(Y)(YX)^i \sum_{j=0}^{\infty} X^{2j}\\
&-X^{e_n+2}(1-Y^{-1}\vep_n X^{-1})\sum_{i=0}^{\frke_{n-1}} b_i(Y)(YX^{-1})^i \sum_{j=0}^{\infty} X^{2j}\\
&=\sum_{l=0}^{\infty} (\sum_{0 \le i \le \frke_{n-1}, j \ge 0 \atop i+2j=l} b_i(Y)Y^i)X^l -Y^{-1}\vep_n\sum_{l=0}^{\infty} (\sum_{0 \le i \le \frke_{n-1}, j \ge 0 \atop i+2j=l-1} b_i(Y)Y^i)X^l \\
&-\sum_{l=0}^{\infty} (\sum_{0 \le i \le \frke_{n-1}, j \ge 0 \atop \frke_n-i+2j=l} b_i(Y)Y^i)X^l+Y^{-1}\vep_n\sum_{l=0}^{\infty} (\sum_{0 \le i \le \frke_{n-1}, j \ge 0 \atop \frke_n-i+2j=l+1} b_i(Y)Y^i)X^l.
\end{align*}
Since $\calg(H;Y,X)$ is a polynomial in $X$ of degree $\frke_n$, the $l$-th coefficient of $\calh(Y,X)$ as a power series in $X$ is
$a_i(Y)$ or $0$ according as $l \le \frke_n$ or $l >\frke_n$, and by a simple computation we prove the assertion.

Let $n$ be odd and. Then we have 
\begin{align*}
\calg(H;Y,X)=  {\calg(H_{n-1};Y,YX) \over 1-\vep_{n-1} X}-\vep_n {X^{\frke_n+1}\calg(H_{n-1};Y,YX^{-1}) \over 1-\vep_{n-1}X}.\tag{B}
\end{align*}
Assume that $\vep_{n-1}\not=0$. Then, as a formal power series in $X$, the right-hand side of (B) can be written as
$$\calg(H_{n-1};Y,YX) \sum_{j=0}^{\infty} (\vep_{n-1} X)^j -\vep_n X^{\frke_n+1}\calg(H_{n-1};Y,YX^{-1}) \sum_{j=0}^{\infty} (\vep_{n-1}X)^j.$$
Then the assertion can be proved in the same manner as above.

Let $n$ be odd, and $\vep_{n-1}=0$. Then  
$$\calg(H;Y,X)=\calg(H_{n-1};Y,YX)+\vep_n X^{\frke_n}\calg(H_{n-1};Y,YX^{-1}).$$
That is, 
\[\calg(H;Y,X)=\sum_{i=0}^{e_{n-1}} b_i (Y)(YX)^i+\vep_nX^{\frke_n}\sum_{i=0}^{e_{n-1}}b_i(Y)(YX^{-1})^i.\]
Thus the assertion directly follows.
\end{proof}

\begin{theorem} \label{th.estimate-of-CF-of-GH}
Let $H=(a_1,\ldots,a_n;\vep_1,\ldots,\vep_n)$ be a naive $\EGK$ datum of length $n$. Let  $q$ be a positive integer.   Then 
\begin{align*} |a_i(q^{1/2})| \le \prod_{l=1}^{n-1} (\frke_{l}+1) \prod_{l=1}^{n-1}
q^{e_ {l}/2} \tag{C}
\end{align*}
for any $i=0,\ldots,e_n$.
\end{theorem}
\begin{proof}
We prove the assertion by induction on $n$. The assertion clearly holds if $n=1$. Let $n \ge 2$ and assume that the assertion holds for any naive $\EGK$ datum of length $n'<n$.
By the functional equation of $\calf(H;X,Y)$, it suffices to prove the assertion for $0 \le i \le e_n/2$.  Put $b_i=b_i(q^{1/2})$. 
Let $n$ be even.  Then, by Theorem \ref{th.induction-formula-for-G}, (1), 
 for $0 \le i \le \frke_n/2$,  $a_i(q^{1/2})$ is given by
\begin{align*}
a_i(q^{1/2}) &=\sum_{\max((i-\frke_{n-1})/2,0) \le j \le i/2} b_{i-2j}q^{(i-2j)/2}  \\
&-\vep_n \sum_{\max(i-1-\frke_{n-1})/2,0) \le j \le (i-1)/2} b_{i-2j-1}q^{(i-2j-2)/2} \\
&-\sum_{0 \le j \le (i-\frke_n+\frke_{n-1}-2)/2} b_{\frke_n+2-2j-i}q^{(\frke_n+2-2j-i)/2} \\
&+\vep_n \sum_{0 \le j \le (i-\frke_n+\frke_{n-1}-1)/2} b_{\frke_n+1-2j-i}q^{(\frke_n-2j-i)/2}.
\end{align*}
Here we make the convention the sum $\sum_{0 \le j \le a} (*) $ is zero if $a<0$.
We also understand $b_j=0$ if $j<0$ or $j >\frke_{n-1}$.
For each $i=0,\ldots,\frke_n/2$, put
\begin{align*} 
\BB_{i,1}&=\#(\{ j \in \ZZ \ | \  \max((i-\frke_{n-1})/2,0) \le j \le i/2 \}),\\
\BB_{i,2}&=\#(\{ j \in \ZZ \ | \  \max((i-1-\frke_{n-1})/2,0) \le j \le (i-1)/2 \}),\\
\BB_{i,3}&=\#(\{ j \in \ZZ \ | \ 0 \le j \le (i-\frke_n+\frke_{n-1}-2)/2 \}),\\
\BB_{i,4}&=\#(\{ j \in \ZZ \ | \ 0 \le j \le (i-\frke_n+\frke_{n-1}-1)/2 \},
\end{align*}
and $\BB_i=\BB_{i,1}+\BB_{i,2}+\BB_{i,3}+\BB_{i,4}$.
By the induction assumption, we have
\[|b_i| \le \prod_{l=1}^{n-2} (\frke_{l}+1) \prod_{l=1}^{n-2}
q^{e_ {l}/2}.\]
Hence we have
\begin{align*}
|a_i(q^{1/2})| & \le  \BB_i q^{\frke_{n-1}/2} \prod_{l=1}^{n-2} (\frke_{l}+1) \prod_{l=1}^{n-2}
q^{e_ {l}/2}
\end{align*}
We claim that we have 
\begin{align*}
\BB_i \le \frke_{n-1}+1. \tag{D}
\end{align*}
To prove this we note that 
\begin{align*}
\#(\{ j \in \ZZ \ | \ \alpha \le j \le \beta) =
\begin{cases} (\beta-\alpha+2)/2 & \text{ if } \alpha,\beta \text { are even} \\
(\beta-\alpha)/2 & \text{ if } \alpha,\beta \text{ are  odd} \\
(\beta-\alpha+1)/2 & \text{ otherwise}
\end{cases} \tag{E}
\end{align*}
for any integers $0 \le \alpha \le \beta$.
Assume $i \ge \frke_{n-1}+1$. Then,  $\max((i-\frke_{n-1})/2,0)=(i-\frke_{n-1})/2$ and $\max((i-1-\frke_{n-1})/2,0)=(i-1-\frke_{n-1})/2$, and by  (E), we easily see that we have $\BB_{i.1}+\BB_{i,2}=\frke_{n-1}+1$.
Moreover, since we have $\frke_{n-1}+1 \le i \le \frke_n/2$, we have
\[(i-\frke_n+\frke_{n-1}-2)/2 <(i-\frke_n+\frke_{n-1}-1)/2 \le (i-\frke_n/2-3)/2<0,\]
and hence $\BB_{i.3}+\BB_{i,4}=0$. This proves (D). Assume that $i \le \frke_{n-1}$. Then, $\max((i-\frke_{n-1})/2,0)=\max((i-1-\frke_{n-1})/2,0)=0$, and   by (E), we have $\BB_{i.1}+\BB_{i,2}=i+1$. Assume that $i-\frke_n+\frke_{n-1}-2 \ge 0$. Then, by (E), we have
$\BB_{i,3}+\BB_{i.4}=i-\frke_n+\frke_{n-1}$, and hence
$\BB_i=2i-\frke_n+\frke_{n-1}+1\le \frke_{n-1}+1$, which proves (D).
Assume that $i-\frke_n+\frke_{n-1}-1 \le 0$. Then, we have $\BB_{i,3}+\BB_{i,4} \le 1$. Moreover, since we have
\[i \le \frke_n-\frke_{n-1}+1 \le \frke_{n-1},\]
and $\frke_n/2$ is an integer, we have $\frke_n/2 \le \frke_{n-1}-1$.
Hence we have
\[\BB_i \le i+2 \le \frke_n/2+2 \le \frke_{n-1}+1.\] 
This completes (D), and hence (C).   

Let $n$ be odd and assume that $\vep_{n-1} \not=0$. Then, by Theorem \ref{th.induction-formula-for-G}, (2.1), 
 for $0 \le i \le \frke_n/2$,  $a_l(q^{1/2})$ is given by
\begin{align*}
a_i(q^{1/2})&=\sum_{\max(i-\frke_{n-1},0) \le j \le i } b_{i-j}q^{(i-j)/2}\\
&-\vep_n \sum_{0 \le j \le \frke_{n-1}+i-\frke_n-1} b_{\frke_n-i-2j+2}q^{(\frke_n-i+j+1)/2}\vep_{n-1}^j.
\end{align*}
Then the assertion can be proved in the same manner as above. 
Finally let $n$ be odd and assume that $\vep_{n-1}=0$. Then, we have
\begin{align*}
\calg(H;q^{1/2},X)=\sum_{i=0}^{\frke_{n-1}} b_i (q^{1/2}X)^i +\vep_n X^{\frke_n}\sum_{i=0}^{\frke_{n-1}} b_i (q^{1/2}X^{-1})^i.
\end{align*}
Then, for $i \le \frke_n/2$,  $a_i(q^{1/2})$ is given by
\begin{align*}
a_i(q^{1/2})=q^{l/2}b_i +\vep_n q^{(\frke_n-i)/2}b_{\frke_n-i}.
\end{align*}
Since $\frke_{n-1}$ is odd, we have  $\frke_{n-1}+1 \ge 2$.
Thus the assertion can be proved by the induction assumption. 
This completes the induction.

\end{proof}

\begin{theorem} \label{th.estimate-of-FH}
Let $H$ be as in Theorem \ref{th.estimate-of-CF-of-GH}. Let $s \in \CC$ such that $\mathrm{Re}(s) \le r_0$.
Then 
\[|\calf(H;q^{1/2},q^s)| \le q^{\frke_n r_0/2} \prod_{i=1}^n (\frke_i+1) \prod_{i=1}^{n-1}
q^{e_ {i}/2}.\]
\end{theorem}
\begin{proof}
We have 
\begin{align*}
\calf(H;q^{1/2},q^s)=\sum_{i=-\frke_n/2}^{-1} a_{i+\frke_n/2}(q^{1/2})q^{si}
+\sum_{i=0}^{\frke_n/2} a_i(q^{1/2})q^{si}.
\end{align*}
Thus the assertion follows from Theorem \ref{th.estimate-of-CF-of-GH}.
\end{proof}
\begin{corollary} \label{cor.estimate-of-FH} Let the notation be as above.
Let $n$ be even. Then 
\[|\calf(H;q^{1/2},q^s)| \le  q^{r_0\frke_n/2}q^{(n-1)\frke_n/4}\prod_{i=1}^n (\frke_i+1) 
.\]

\end{corollary}
\begin{proof}
By definition,
\begin{align*}
\sum_{i=1}^{n-1} \frke_{i} \le \sum_{i=1}^{n-1} (n-i)a_i. 
\end{align*}
For $\underline a=(a_1,\ldots,a_n)$, put $|\underline a|=a_1+\cdots+a_n$. 
First suppose that $|\underline a|$ is even, then, by definition, we have
$\frke_n=|\underline a|$, and hence we have 
\begin{align*}
&2\sum_{i=1}^{n-1} (n-i)a_i -(n-1) \frke_n =\sum_{i=1}^n (n-2i+1)a_i \\
& = \sum_{i=1}^{n/2}(n-2i+1)(a_i-a_{n-i+1}) \le 0.
\end{align*}
This implies that we have 
\begin{align*}
\sum_{i=1}^{n-1} (n-i)a_i \le \frac{n-1} 2 \frke_n,\end{align*}
and hence 
\begin{align*}
\sum_{i=1}^{n-1} \frke_{i} \le  \frac{n-1} 2 \frke_n.
\end{align*}
Next suppose that $|\underline a|$ is odd, then, again by definition, we have
$\frke_n=|\underline a|-1$, and $a_n \ge a_1+1$.
Hence we have
\begin{align*}
&2\sum_{i=1}^{n-1} (n-i)a_i -(n-1) \frke_n =\sum_{i=1}^n (n-2i+1)a_i +(n-1)\\
& = \sum_{i=2}^{n/2}(n-2i+1)(a_i-a_{n-i+1}) +(n-1)(a_1+1-a_n) \le 0.
\end{align*}
Thus the assertion can be proved in the same manner as above.

\end{proof}
\begin{theorem} \label{th.H-and-Siegel-series} 
Let $q=\#(\frko/\frkp)$. Let $B \in \calh_n(\frko)^{\mathrm{nd}}$. Then 
there is a naive $\mathrm{EGK}$ datum $H=(a_1,\ldots,a_n;\vep_1,\ldots,\vep_n)$ of length $n$ such that 
$(a_1,\ldots,a_n)$ is the Gross-Keating invariant of $B$ and
\[\calf(H;q^{1/2},X)=\widetilde F_{\frkp}(B,X).\]
\end{theorem}
\begin{proof}
The assertion follows from \cite[Theorem 4.3]{IK18}, \cite[Corollary 5.1]{IK18}, \cite[Theorem. 1.1]{IK22}, and \cite[Proposition 4.5]{IK22}. 
\end{proof}
By Corollary \ref{cor.estimate-of-FH} and Theorems \ref{th.estimate-of-FH} and \ref{th.H-and-Siegel-series}, we immediately have the following theorem.
\begin{theorem}\label{th.estimate-of-F}
Let $B \in \calh_n(\frko)^{\mathrm{nd}}$ with $n$ even and $s \in \CC$ such that $\mathrm{Re}(s) \le r_0$. Let  $\frke_i=\frke_i(B)$ be that in Section 4. Then we have the following estimates.
\begin{itemize}
\item[(1)] 
\begin{align*}
|\widetilde F_\frkp(B,\alpha)| \le q^{r_0\frke_n/2}q^{(n-1)\frke_n/4} \prod_{i=1}^n (\frke_i +1).
\end{align*}
\item[(2)] \begin{align*}
|\widetilde F_\frkp(B,\alpha)| \le q^{r_0\frke_n/2} \prod_{i=1}^{n-1}
q^{e_ {i}/2}.\prod_{i=1}^n (\frke_i +1).
\end{align*}
\end{itemize}
\end{theorem}

\section{Proofs of Theorem \ref{th.main-result} and \ref{th.main-result2}}
In this section we prove our main result. 
\begin{lemma} \label{lem.Conrey-Iwaniec}
Let $\kappa \in {1 \over 2} + \ZZ$ such that $\kappa \ge 13/2$.
Let $g(z)=\sum_{m=1}^\infty c_g(h){\bf e}(mz) \in S_\kappa(\varGamma_0(4))$. Then, for any fundamental discriminant $D$ we have
\begin{align*}
c_g(|D|) \ll_{\vep,g} |D|^{\kappa/2-1/3+\vep} \ (\vep>0).
\end{align*}
\end{lemma}
\begin{proof} It is known that $D$ can be expressed as $D=2^sD'$,
with $s=0,2$ or $3$ and $D'$ a square free odd integer. Suppose that $s=0$. Then, by \cite[Corollary 1.3]{CI00}, we have
\begin{align*}
c_g(|D|) \ll_{\vep, g} |D|^{\kappa/2-1/3+\vep} \ (\vep>0).
\end{align*}
Suppose that $s \ge 2$. Let $T_{\kappa,1}^4(4)$ be the  operator on $S_{\kappa}(\varGamma_0(4))$ in \cite[page 450]{Sh73}.  Let
\[\widetilde g(z):=g|T_{\kappa,1}^4(4)(z)=\sum_{m=1}^\infty b(m){\bf e}(mz).\]
Then, by Theorem \cite[Theorem 1.7]{Sh73}, we have
$c_g(|D|)=b(2^{s-2}|D'|)$. Since  $\widetilde g$ belongs to $S_\kappa(\varGamma_0(4))$ and $2^{s-2}|D'|$ is  square free,  again by \cite[Corollary 1.3]{CI00}, we have
\begin{align*}
c_g(|D|) =b(2^{s-2}|D'|) \ll_{\vep,\widetilde g} |2^{s-2}D'|^{\kappa/2-1/3+\vep} \ll_{\vep,g} |D|^{\kappa/2-1/3+\vep} \ (\vep >0).
\end{align*}
This completes the proof.
\end{proof}
\begin{remark}\label{rem.unstability-of-Kohnen-plus-space}
For $g \in S_\kappa^+(\varGamma_0(4))$, $g|T_{\kappa,1}^4(4)$ does not necessarily belong to $S_\kappa^+(\varGamma_0(4))$. 
\end{remark}
For an element $C \in \calh_r(\ZZ)_{>0}$, put
\[\boldsymbol{\Delta}(C)=\begin{cases} \frkf_C^2 & \text{ if } r \text{ is even} \\
2^{r-1}\det C & \text{ if } r \text{ is odd.}
\end{cases}\]
Let $B =(b_{ij}) \in \calh_n(\ZZ)_{>0}$. For each prime number $p$, we denote by  $\GK(B)^{(p)}=\underline a^{(p)}=(a_1^{(p)},\ldots,a_n^{(p)})$ the Gross-Keating invariant of $B$ viewing $B$ as an element of $\calh_n(\ZZ_p)$, and for $r=1,\ldots,n$ put
\[\frke_{r,B}^{(p)}=\begin{cases}
2[(a_1^{(p)}+\cdots+a_r^{(p)})/2] & \text{ if } r \text{ is even}\\
a_1^{(p)}+\cdots+a_r^{(p)} & \text{ if } r \text{ is odd},
\end{cases}\]
and in particular put $\frke_B^{(p)}=\frke_{n,B}^{(p)}$.
Moreover, for $r \le n$, put
\[\mathscr{E}_r(B)=\mathrm{GCD}_{X \in M_{nr}(\ZZ) \atop  B[X] >0}
\boldsymbol{\Delta}(B[X]),\]
\[D_r(B)=\mathrm{GCD}_{X \in M_{nr}(\ZZ) \atop  B[X] >0} 2^{2[r/2]} \det B[X].\]
Clearly $\scre_r(B)$ divides $D_r(B)$,
\begin{lemma}\label{lem.estimate-of-GGK}
Let $B \in \calh_n(\ZZ)_{>0}$ with $n$ even.  
\begin{itemize}
\item[(1)] For any positive integer $r \le n$, the product 
$\prod_{p | \frkf_B} p^{\frke_{r,B}^{(p)}} $ divides $\scre_r(B)$.
In particular, we have
\[\prod_{p \frkf_B} p^{\frke_{n,B}^{(p)}} =\frkf_B^2.\]
\item[(2)] For any positive integer $r \le n$, 
$G_r(B)=D_r(B)$.
\end{itemize}
\end{lemma}
\begin{proof}
The assertion (1) follows from Theorems \ref{th.GK-invariant}, \ref{th.estimate-of-GK}. We prove (2). By Lemma \ref{lem.det-of-H[X]}, $G_r(B)$ divides $D_r(B)$.
To prove that  $D_r(B)$ divides $G_r(B)$, for 
each prime number $p$, we denote by ${\bf g}_r(B)^{(p)}$ and ${\bf d}_r(B)^{(p)}$ the quantities  ${\bf g}_r(B)$ and ${\bf d}_r(B)$ defined in Section 3, respectively, viewing $B$ as an element of $\calh_n(\ZZ_p)^{\rm nd}$.
Let $X \in M_{nr}(\ZZ_p)$ and suppose that $\det B[X] \not=0$. Then we can take $X_0 \in M_{nr}(\ZZ)$ such that $X_0 \equiv X \text{ mod } p^eM_{nr}(\ZZ_p)$ with $e > \ord_p(\det B[X])$. Then, by definition, we have
\[\ord_p(D_r(B)) \le \ord_p(\det B[X_0])=\ord_p(\det B[X]).\]
By Theorem \ref{th.estimate-of-GK2}, this implies that 
\[\ord_p(D_r(B)) \le {\bf d}_r(B)^{(p)} ={\bf g}_r(B)^{(p)},\]
and hence 
$D_r(B)$ divides $\prod_p p^{{\bf g}_r(B)^{(p)}}=G_r(B)$.
This proves the assertion (2). 
\end{proof}

The following theorem is a refined version of Theorem \ref{th.main-result}.
\begin{theorem}\label{th.refined-estimate}
Let the notation and the assumption be as in Theorem \ref{th.main-result}.
Let $B \in \calh_n(\ZZ)_{>0}$. The we have the following estimates.
\begin{itemize}
\item[(1)] We have
\begin{align*}
&|c_{I_n(h)}(B)| \le |c_h(|\frkd_B|)|\frkf_B^{k-1}
\prod_{i=1}^n \prod_{p | \frkf_B}(1+\frke_{i,B}^{(p)})).
\end{align*}
\item[(2)]  We have
\begin{align*}
&|c_{I_n(h)}(B)| \le |c_h(|\frkd_B|)|\frkf_B^{k-(n+1)/2}
\prod_{i=1}^{n-1} \scre_i(B)^{1/2} \prod_{i=1}^n \prod_{p | \frkf_B}(1+\frke_{i,B}^{(p)})),
\end{align*}
and in particular we have
\begin{align*}
&|c_{I_n(h)}(B)| \le |c_h(|\frkd_B|)|\frkf_B^{k-(n+1)/2}
\prod_{i=1}^{n-1} G_i(B)^{1/2} \prod_{i=1}^n \prod_{p | \frkf_B}(1+\frke_{i,B}^{(p)})).
\end{align*}

\end{itemize}
\end{theorem}
\begin{proof}
From now on we write $\widetilde F_p(B,X)$ instead of $\widetilde F_\frkp(B,X)$ if $\frkp=(p)$.
We have
\begin{align*}
c_{I_n(h)}(B)=c_h(|\frkd_B|)\frkf_B^{k-(n+1)/2} \prod_{p | \frkf_B} \widetilde F_p(B,\alpha_p).
\end{align*}
We have $|\alpha_p|=1$ for any prime number $p$. Hence, by Theorem \ref{th.estimate-of-F} (1) and  Lemma \ref{lem.estimate-of-GGK}, we have 
\begin{align*}
&|\prod_{p | \frkf_B} \widetilde F_p(B,\alpha_p)| \\
&\le  \prod_{p | \frkf_B}p^{(n-1)\frke_B^{(p)}/4} \prod_{i=1}^n \prod_{p | \frkf_B}(1+\frke_{i,B}^{(p)}))\\
&=\frkf_B^{(n-1)/2}\prod_{i=1}^n \prod_{p | \frkf_B}(1+\frke_{i,B}^{(p)})).
\end{align*}
This proves the assertion (1). Similarly, by Theorem \ref{th.estimate-of-F} (2), we have 
\begin{align*}
&|\prod_{p | \frkf_B} \widetilde F_p(B,\alpha_p)| \\
&\le  \frkf_B^{k-(n+1)/2}\prod_{i=1}^{n-1} \prod_{p | \frkf_B} p^{\frke_{i,B}^{(p)}/2} \prod_{i=1}^n \prod_{p | \frkf_B}(1+\frke_{i,B}^{(p)})).
\end{align*}
By Lemma \ref{lem.estimate-of-GGK} (1) and the remark before it, we have
\[\prod_{p | \frkf_B} p^{\frke_{i,B}^{(p)}/2} \le \scre_i(B)^{1/2} \le D_i(B)^{1/2}=G_i(B)^{1/2}.\]
This proves the assertion (2).

\end{proof}

{\bf Proofs of Theorems \ref{th.main-result} and \ref{th.main-result2}}

 By Theorem \ref{th.GK-invariant}, we have
$\frke_B^{(p)}=2\ord_p(\frkf_B)$ and $e_{i,B} \le \ord_p(\det (2B))$ for any $i=1,\ldots,n$.
Hence we have
\[\prod_{p | \frkf_B} (1+\frke_{i,B}^{(p)}) \le d(\det (2B))\]
for any $i=1,\ldots,n$, where $d(a)$ is the number of positive divisors of $a$ for $a \in \ZZ_{>0}$.
Hence, by Theorem \ref{th.refined-estimate} (1), we have
\begin{align*}
|c_{I_n(h)}(B)| & \le |c_h(|\frkd_B|)|\frkf_B^{k-1} d(\det (2B))^n\\
&= |c_h(|\frkd_B|)|\frkd_B|^{-k/2 +1/2}d(\det (2B))^n\det (2B)^{(k-1)/2}\\
\end{align*}
By Lemma \ref{lem.Conrey-Iwaniec} we have 
\begin{align*}
c_h(|\frkd_B|) \ll_{\vep,h} |\frkd_B|^{k/2-n/4-1/12+\vep} \quad (\vep >0)
\end{align*}
Hence we have
\begin{align*}
&c_{I_n(h)}(B) \\
&\ll_{\vep, I_n(h)} |\frkd_B|^{-n/4+5/12+\vep} 
d(\det (2B))^n(\det (2B))^{(k-1)/2},
\end{align*}

We have  $|\frkd_B|^{\vep} \le \det (2B)^{\vep}$ and 
\begin{align*}
d(\det (2B))^n  \ll_{\vep} (\det (2B))^\vep \quad (\vep >0)
\end{align*}
for any $B \in \calh_n(\ZZ)_{>0}$. 
Thus we complete the proof of Theorem \ref{th.main-result}.
Similarly, we can prove Theorem \ref{th.main-result2}.

\end{document}